\documentclass[final,1p,times]{elsarticle}

\usepackage{amsthm,amsmath,amssymb,mathrsfs,amsfonts,graphicx,graphics,latexsym,exscale,cmmib57,dsfont,amscd,ulem}
\usepackage{mathrsfs}
\usepackage{fancybox}
\usepackage{color,soul}
\usepackage[colorlinks=true]{hyperref}

\newcommand{\X}{\boldsymbol{X}}
\newcommand{\Y}{\boldsymbol{Y}}
\newcommand{\bH}{\boldsymbol{H}}

\newcommand{\e}{\boldsymbol{o}}

\newcommand{\bea}{\begin{eqnarray}}
\newcommand{\eea}{\end{eqnarray}}
\newcommand{\bean}{\begin{eqnarray*}}
\newcommand{\eean}{\end{eqnarray*}}

\newtheorem{cor-of-0.11}{Corollary}
\newtheorem{Thm}{Theorem}[section]
\newtheorem{cor}[Thm]{Corollary}
\newtheorem{prop}[Thm]{Proposition}
\newtheorem{Lem}[Thm]{Lemma}
\newtheorem{NLem}[Thm]{Notions and Lemma}
\newtheorem*{cor-of-0.13}{Corollary}

\theoremstyle{definition}

\newtheorem{defn}[Thm]{Definition}

\numberwithin{equation}{section}

\journal{xxx}

\begin{document}
\begin{frontmatter}

\title{An extension of Furstenberg's structure theorem for Noetherian modules and multiple recurrence theorems II}

\author{Xiongping Dai}
\ead{xpdai@nju.edu.cn}
\address{Department of Mathematics, Nanjing University, Nanjing 210093, People's Republic of China}

\begin{abstract}
Using a recent Furstenberg structure theorem, we obtain a quantitative multiple recurrence theorem relative to any locally compact second countable Noetherian module over a syndetic ring.
\end{abstract}

\end{frontmatter}


\section*{Introduction}\label{sec0}
This paper will be devoted to studying the Furstenberg multiple recurrence of dynamical systems induced by any locally compact second countable Noetherian module over a syndetic ring acting on a standard Borel probability space, as a subsequent work of \cite{Dai-pre}.

First of all, we recall that by an ``lcscN'' $R$-module $G$ over a ``syndetic'' ring $(R,+,\cdot)$, we mean that $G$ and $(R,+)$ both are locally compact second countable Hausdorff commutative groups such that:
\begin{itemize}
\item The operation $(t,g)\mapsto t\cdot g$ is continuous from $R\times G$ to $G$.
\item $G$ is a Noetherian $R$-module: for every sequence $G_0\subseteq G_1\subseteq G_2\subseteq\dotsm$ of $R$-submodules of $G$, we have $G_n=G_{n+1}$ as $n$ sufficiently large.
\item $R$ is syndetic: $\forall t\not=0$, $Rt$ is syndetic in the sense that one can find a compact subset $K$ of $R$ with $K+Rt=R$.
\end{itemize}
See \cite{Dai-pre}. Clearly, $(\mathbb{Z}^n,+)$ over $\mathbb{Z}$, $(\mathbb{Q}^n,+)$ over $\mathbb{Q}$, and $(\mathbb{R}^n,+)$ over $\mathbb{R}$ all are lcscN modules over syndetic rings. Moreover, the $p$-adic integer (syndetic) ring $\mathbb{Z}_p$ and the (syndetic) $p$-adic number field $\mathbb{Q}_p$ both are lcscN as modules over themselves. See \cite{Lan} for more examples.

Let $G$ be an lcsc $R$-module and let $(X,\mathscr{X},\mu)$ be a standard Borel probability space. We will consider a measure-preserving Borel $G$-action dynamical system:
\begin{gather*}
T\colon G\times X\rightarrow X\quad \textrm{or write}\quad G\curvearrowright_TX\label{eq0.1}
\end{gather*}
where $T_g\colon X\rightarrow X$ is $\mu$-preserving for each $g\in G$ and the $G$-action map $T\colon (g,x)\mapsto T_g(x)$ is jointly measurable from $G\times X$ to $X$. For brevity, we will write
\begin{gather*}
T(t\cdot g,x)=T_{t\cdot g}(x)=T_g^t(x)=g^tx,\quad \forall t\in R\textit{ and }g\in G,
\end{gather*}
if no confusion. Then given any $g\not=e$ the identity element of $G$, \
\begin{gather*}
T_g\colon R\times X\rightarrow X\quad  \textit{by }\ (t,x)\mapsto T_g^t(x)=g^tx
\end{gather*}
defines a new $\mu$-preserving Borel $R$-system.

A metatheorem of dynamical theory states that whenever the underlying space of the dynamical system is appropriately bounded in the sense of topology or measure theory, the orbits of the motion will necessarily exhibit some form of recurrence, or return close to their initial position. The first precise theorem of this kind was formulated by H.~Poincar\'{e} 1890~\cite{Poi}.

A quantitative recurrence theorem is due to A.\,Y.~Khintchine as follows. Let $\mathbb{R}$ be equipped with the standard Euclidean topology. Then Khintchine's Recurrence Theorem says that:
\begin{itemize}
\item If $T\colon\mathbb{R}\times X\rightarrow X$ is a $C^0$-flow on a compact metric space $X$ preserving a Borel probability measure $\mu$, then for any $E\in\mathscr{B}_X$ with $\mu(E)>0$ and $\varepsilon>0$, the set
    \begin{gather*}
    \left\{t\in\mathbb{R}\,|\,\mu(E\cap T^{-t}E)>\mu(E)^2-\varepsilon\right\}
    \end{gather*}
    is relatively dense in $\mathbb{R}$ (cf.~\cite{Khi, NS}).
\end{itemize}
Poincar\'{e}'s Recurrence Theorem says that $\mu(E\cap T^{-t}E)>0$ for infinitely many values of $t\in\mathbb{R}$; but the Khintchine Recurrence Theorem says that $\mu(E\cap T^{-t}E)$ is `large' for `many' values of $t\in\mathbb{R}$. Khintchine's theorem has been extended from $\mathbb{R}$-actions to $\sigma$-compact amenable group actions in \cite{Dai-16}. Notably V.~Bergelson, B.~Host and B.~Kra have recently derived in 2005 \cite{BHK} the following multiple version of Khintchine's theorem for any \textit{$\mu$-ergodic} $\mathbb{Z}$-action $T\colon \mathbb{Z}\times X\rightarrow X$:
\begin{itemize}
\item  Each of the following two returning-time sets, for any $E\in\mathscr{B}_X$ with $\mu(E)>0$ and $\varepsilon>0$, is relatively dense in $\mathbb{Z}$:
\begin{gather*}
\left\{n\in\mathbb{Z}\,|\,\mu(E\cap T^{-n}E\cap T^{-2n}E)>\mu(E)^3-\varepsilon\right\}\intertext{and}
\left\{n\in\mathbb{Z}\,|\,\mu(E\cap T^{-n}E\cap T^{-2n}E\cap T^{-3n}E)>\mu(E)^4-\varepsilon\right\}.
\end{gather*}
See \cite[Theorem~1.2]{BHK}.

\item Also see \cite{Fra, Pot} for polynomials version of the above multiple Khintchine theorem of Bergelson et al.
\end{itemize}

It should be noted that Bergelson et al have proven that the multiple version of Khintchine's theorem does not need to hold for \textit{non-ergodic} $\mathbb{Z}$-actions (even for the $2$-multiple case $\{n,2n\}$, cf.~\cite[Theorems~2.1]{BHK}) and does not need to be true for more \textit{higher multiple} case ($\{n,2n,3n, 4n\}$, cf.~\cite[Theorem~1.3]{BHK}) by constructing counterexamples based on F.A.~Behrend's and I.~Ruzsa's combinatorial results.

However, by using completely different technique framework, we can obtain the following theorem, which claims any measure-preserving standard Borel $G$-system always has a nontrivial factor that possesses the multiple Khintchine recurrence. That is the following

\begin{Thm}[Multiple Khintchine Recurrence]\label{Kh}
Let $G$ be an lcscN $R$-module over a syndetic ring $(R,+,\cdot)$ and let $(X,\mathscr{X},\mu)$ be a nontrivial standard Borel $G$-space. Then there exists a $G$-invariant $\sigma$-subalgebra $\mathscr{X}_1$ of $\mathscr{X}$ with $\{\varnothing,X\}\not=\mathscr{X}_1$ ($\mu$-$\mathrm{mod}$ $0$) such that for any
$A\in\mathscr{X}_1$ with $\mu(A)>0$, any $g_1,\dotsc,g_l\in G$, $l\ge2$, and any $\varepsilon>0$, the set of returning times of big-hitting
\begin{gather*}
N_{g_1,\dotsc,g_l}(A,\varepsilon)=\left\{t\in R\,\big{|}\,\int_Xg_1^t1_A\dotsm g_l^t1_A\,d\mu>\mu(A)^l-\varepsilon\right\},
\end{gather*}
is of positive lower Banach density in $(R,+)$; i.e.,
\begin{gather*}
\liminf_{n\to\infty}\frac{|F_n\cap N_{g_1,\dotsc,g_l}(A,\varepsilon)|}{|F_n|}>0
\end{gather*}
over any weak F{\o}lner sequence $\{F_n\}_{1}^\infty$ in $(R,+)$.
Consequently it is syndetic in $(R,+)$.
\end{Thm}

Here a sequence of nonull compact subsets $F_n$ of positive Haar measure in $(R,+)$ with a Haar measure $|\cdot|$, which is such that
\begin{equation*}
\lim_{n\to\infty}\frac{|(r+F_n)\vartriangle F_n|}{|F_n|}=0\quad\forall r\in R,
\end{equation*}
is called a \textit{weak F{\o}lner sequence} in $(R,+)$.\footnote{Here and in the future, unlike the case in literature, we neither need to require
\begin{equation*}
\lim_{n\to\infty}\frac{|(K+F_n)\vartriangle F_n|}{|F_n|}=0\quad\forall K\subset R\textit{ compact}
\end{equation*}
nor other regularities such as Tempelman condition or Shulman condition (cf.~\cite{EW}).}

\subparagraph{Outline of the proof of Theorem~\ref{Kh}}
For example, $(\mathbb{R}^n,+)$ is an lcscN $\mathbb{R}$-module but it is not a free abelian group. So our Theorem~\ref{Kh} is already beyond the framework of Furstenberg-Katznelson \cite{FK,Fur}.
To prove Theorem~\ref{Kh} in the (not necessarily ergodic) probabilistic settings, our main tool is the following recent structure theorem.

\begin{Thm}[{Furstenberg Structure Theorem~\cite{Dai-pre}}]\label{thm0.2}
Let $G$ be an lcscN $R$-module over a syndetic ring $(R,+,\cdot)$. Then for any nontrivial standard Borel $G$-system $\X=(X,\mathscr{X},\mu,G)$, there exists an ordinal $\eta$ and a system of $G$-factors $\left\{\pi_\xi\colon\X\rightarrow\X_\xi\right\}_{\xi\le\eta}$ such that:
\begin{enumerate}
\item[$(a)$] $\X_0$ is the one-point $G$-system and $\X_\eta=\X$ ($\mu$-$\mathrm{mod}$ $0$).
\item[$(b)$] If $0\le\theta<\xi\le\eta$, then there is a factor $G$-map $\pi_{\xi,\theta}\colon\X_\xi\rightarrow\X_\theta$ with $\pi_\theta=\pi_{\xi,\theta}\circ\pi_\xi$.
\item[$(c)$] For each ordinal $\xi$ with $0\le\xi<\eta$, $\pi_{\xi+1,\xi}\colon\X_{\xi+1}\rightarrow\X_\xi$ is a nontrivial ``primitive'' extension.
\item[$(d)$] If $\xi$ is a limit ordinal $\le\eta$, then $\X_\xi=\underleftarrow{\lim}_{\theta<\xi}\X_\theta$.
\end{enumerate}
Moreover, the intermediate factors are of the form
\begin{gather*}
\X_\xi=(X,\mathscr{X}_\xi,\mu,G),\quad \pi_\xi=\textit{Id}_X\quad \textit{and}\quad \pi_{\xi+1,\xi}=\textit{Id}_X\quad (0<\xi<\eta).
\end{gather*}
Here we refer to
\begin{gather*}
\X\rightarrow\X_\eta\rightarrow\dotsm\rightarrow\X_{\xi+1}\rightarrow\X_\xi\rightarrow\dotsm\rightarrow\X_1\rightarrow\X_0
\end{gather*}
a ``Furstenberg factors chain'' of $\X$.
\end{Thm}

Since $\X_0$ is just the one-point $G$-system, them the first intermediate factor $\X_1=(X,\mathscr{X}_1,\mu,G)$ in the Furstenberg factors chain of $\X$ is such that $(X,\mathscr{X}_1,\mu,G_c)$ is compact and $(X,\mathscr{X}_1,\mu,G_w)$ is totally weak-mixing, where $G=G_c\times G_w$ is some direct product of two $R$-submodules of $G$. Then we only need to show $\X_1$ is a \textit{Kh}-system. See Proposition~\ref{prop4.1} in $\S\ref{sec4}$.

The remainder of this paper will be organized as follows:

\tableofcontents
\subsection*{Acknowledgments}
Finally, the author is deeply grateful to Professor Hillel Furstenberg for many helpful suggestions, comments, and carefully checking the details of the original manuscript.

\section{A van der Corput-type lemma}\label{sec1}
In this section we will prove a technical tool\,---\,a van der Corput-type lemma\,---\,for proving the \textit{Kh}-property of a totally weak-mixing system in $\S\ref{sec3}$.
\begin{itemize}
\item Let $(R,+)$ be an lcsc abelian group with the zero element $\e$ and with a fixed Haar measure $|\centerdot|$ or $dt$ such that $|R|=\infty$ in the sequel of this section. Here $R$ is not necessarily a topological ring.
\end{itemize}
We can then take a weak F{\o}lner sequence $\{F_n\}_1^\infty$ in $(R,+)$; that is, $\{F_n\}_1^\infty$ is a sequence of compact subsets of positive Haar measure of $R$ satisfying the so-called \textit{weak F{\o}lner condition}~\cite{Pat}:
\begin{gather*}
\lim_{n\to\infty}\frac{|(r+F_n)\vartriangle F_n|}{|F_n|}=0\quad \forall r\in R.
\end{gather*}
Given any measurable subset $S$ of $R$, its \textit{upper density over $\{F_n\}_1^\infty$} is defined as follows:
$$
\mathrm{D}^*(S)=\limsup_{n\to\infty}\frac{|S\cap F_n|}{|F_n|}.
$$
If $\mathrm{D}^*(S)=0$  over $\{F_n\}_1^\infty$ then we say that $S$ has \textit{density $0$ over $\{F_n\}_1^\infty$}. By the weak F{\o}lner condition we note that over the same weak F{\o}lner sequence $\{F_n\}_1^\infty$ in $(R,+)$,
\begin{gather*}
\mathrm{D}^*(S)=\mathrm{D}^*(S+r)\quad \forall r\in R.
\end{gather*}
We can similarly define the \textit{lower density} $\mathrm{D}_*(S)$ of $S$ over $\{F_n\}_1^\infty$ by replacing the $\limsup$ by $\liminf$ and further define \textit{density} $\mathrm{D}(S)$ over $\{F_n\}_1^\infty$.

We shall speak of convergence in density over $\{F_n\}_1^\infty$ for a family of points but for a subfamily of $0$-density. The following is the precise definition.

\begin{defn}\label{def1.1}
Let $(R,+)$ be an lcsc abelian group and $\{F_n\}_{1}^\infty$ a weak F{\o}lner sequence in it.
\begin{enumerate}
\item[(1)] For a Borel measurable function $x_\centerdot\colon R\rightarrow X$ from $R$ into a topological space $X$, we say that \textit{$x_\centerdot$ converges in density over $\{F_n\}_1^\infty$ to a point $x\in X$} if for every neighborhood $V$ of $x$ in $X$, $x_t\in V$ but for a set of $t$ of $\{F_n\}_1^\infty$-density $0$. We write
    \begin{gather*}\textrm{D-}{\lim}_{t\in R}x_t=x\quad \textrm{over }\{F_n\}_1^\infty.\end{gather*}

\item[(2)] Let $a_{\centerdot}\colon R\rightarrow\mathbb{R}$ be a measurable function. The ess.l.u.b. (possibly infinite)
\begin{gather*}
\textrm{D-}{\limsup}_{t\in R}a_t:=\inf\big{\{}c\colon \mathrm{D}^*(\{t\in R\,|\,a_t>c\})=0~\textrm{over}~\{F_n\}_1^\infty\big{\}}
\intertext{or equivalently}
\textrm{D-}{\limsup}_{t\in R}a_t=\inf\big{\{}c\colon \mathrm{D}(\{t\in R\,|\,a_t\le c\})=1~\textrm{over}~\{F_n\}_1^\infty\big{\}}
\end{gather*}
is called the \textit{$\limsup$ of $a_\centerdot$} or \textit{$(a_t)_{t\in R}$ in density over $\{F_n\}_1^\infty$}.
\item[(3)] We can similarly define $\textrm{D-}\liminf_{t\in R}a_t$ for any $a_{\centerdot}\colon R\rightarrow\mathbb{R}$ by using ess.g.l.b.
\end{enumerate}
\end{defn}
\noindent
Clearly, if $a_\centerdot$ or $(a_t)_{t\in R}$ is bounded, then
$$-\infty<\textrm{D-}{\liminf}_{t\in R}a_t\le\textrm{D-}{\limsup}_{t\in R}a_t<\infty$$
over any $\{F_n\}_1^\infty$. In addition, whenever $a_\centerdot\ge0$ and $\textrm{D-}\limsup_{t\in R}a_t=0$, then $\textrm{D-}\lim_{t\in R}a_t=0$ over $\{F_n\}_1^\infty$. Note that $\textrm{D-}\lim_{t\in R}a_t$ is completely different with the Moore-Smith limit involving a directed system. See \cite[Sections~4.2 and 7.2]{Fur} for the special case of $R=\mathbb{Z}$ or $\mathbb{Z}_+$ with the discrete topology.

We shall need the following notions and lemma.

\begin{NLem}\label{lem1.2}
Let $\mathcal{F}_+$ be the family of all the Borel subsets of $R$ with positive upper density over the same weak F{\o}lner sequence $\{F_n\}_1^\infty$ in $(R,+)$. Then
\begin{itemize}
\item $\mathcal{F}_+$ has the ``Ramsey property'' (cf.~\cite[Def.~9.1]{Fur} for $R=\mathbb{Z}$), i.e., if $S_1\cup S_2\in\mathcal{F}_+$ and $S_1,S_2\in\mathscr{B}_R$ then either $S_1\in\mathcal{F}_+$ or $S_2\in\mathcal{F}_+$;
\item $\mathcal{F}_+$ has the ``van der Corput property'', i.e., for $S\in\mathcal{F}_+$, $\left\{q\in R\,|\,(S-q)\cap S\in\mathcal{F}_+\right\}\in\mathcal{F}_+$.
\end{itemize}
We define the dual family $\mathcal{F}_+^*$ to consist of all those Borel subsets of $R$ that intersect non-voidly each member of $\mathcal{F}_+$. Then
$\mathcal{F}_+^*=\left\{S\in\mathscr{B}_R\,|\,\mathrm{D}(S)=1\textrm{ over }\{F_n\}_1^\infty\right\}$.
\end{NLem}

\begin{proof}
The Ramsey property of $\mathcal{F}_+$ trivially holds and the van der Corput property follows clearly from the translation invariance of upper/lower densities over any given weak F{\o}lner sequence $\{F_n\}_1^\infty$ in $(R,+)$.
\end{proof}

We shall need the following lemma, which is a generalization of \cite[Lemma~4.8]{Fur} for the special case that $R=\mathbb{Z}$.

\begin{Lem}\label{B}
Let $Q\subset R$ be a Borel subset of density $1$ over $\{F_n\}_1^\infty$ and for each $q\in Q$ let $R_q$ be a Borel subset of $R$ of density $1$ over $\{F_n\}_1^\infty$. Let $S$ be a Borel subset of $R$ with $\mathrm{D}^*(S)>0$ over $\{F_n\}_1^\infty$ and let $k\ge1$ be any given integer. Then there exist $k$ distinct points $r_1,\dotsc,r_k$ in $S$ such that each $r_j-r_i\in Q$ and $r_i\in R_{r_j-r_i}$ for $1\le i<j\le k$.
\end{Lem}

\begin{proof}
As in the $\mathbb{Z}$ case~\cite[Lemma~4.8]{Fur}, we shall show by induction that there exists a Borel set $S_k\subset S$ of positive upper density over $\{F_n\}_1^\infty$, and a set of $k$ points $t_1,\dotsc,t_k$ in $R$ such that for each $r\in S_k$ the $k$-tuple $r_1=r+t_1,\dotsc,r_k=r+t_k$ has the properties sought in the lemma.

For $k=1$ there is nothing needed to prove by letting $S_1=S, t_1=\e$; so suppose that $S_k$ and $\{t_1,\dotsc,t_k\}$ have been found. We shall find $t_{k+1}$ such that $t_{k+1}-t_i\in Q$ for $1\le i\le k$, and a Borel subset $S_{k+1}$ of $S_k$ with positive upper density over $\{F_n\}_1^\infty$ such that for each $r\in S_{k+1}$ the additional conditions, $r+t_{k+1}\in S$ and $r+t_i\in R_{t_{k+1}-t_i}, 1\le i\le k$, are satisfied.

We first note that since $S_k$ has positive upper density over $\{F_n\}_1^\infty$, then by Lemma~\ref{lem1.2}
$$
S_k^*=\big{\{}t\in R\colon \mathrm{D}^*((S_k-t)\cap S_k)>0~\textrm{over }\{F_n\}_1^\infty\big{\}}
$$
is of positive upper density over $\{F_n\}_1^\infty$. We now choose $t_{k+1}$ in $S_k^*\cap\bigcap_{i\le k}(Q+t_i)$. We can then obtain $t_{k+1}-t_i\in Q$ for $1\le i\le k$. Finally we set
$$
S_{k+1}=(S_k-t_{k+1})\cap S_k\cap\bigcap_{i\le k}(R_{t_{k+1}-t_i}-t_i).
$$
This set has positive upper density over $\{F_n\}_1^\infty$ since $(S_k-t_{k+1})\cap S_k$ does so and $\bigcap_{i\le k}(R_{t_{k+1}-t_i}-t_i)$ has density $1$ over $\{F_n\}_1^\infty$.

The proof of Lemma~\ref{B} is thus completed.
\end{proof}

Motivated by \cite[Lemmas~4.9, 7.5 and 9.24]{Fur}, \cite[Theorem~4.6]{B} and \cite[Theorem~7.11]{EW}, by Lemma~\ref{B} we can obtain the following useful technical result, which is another version of the classical van der Corput lemma~\cite{Cor}.

\begin{Lem}\label{lem1.4}
Let $x_\centerdot\colon R\rightarrow\bH$ be a bounded Borel measurable function, with $\|x_t\|\le\beta\ \forall t\in R$, from an \textit{lcsc} abelian $(R,+)$ into a Hilbert space $\bH$ with the inner product $\langle\cdot,\cdot\rangle$ and suppose that
\begin{gather*}
\textrm{$\mathrm{D}$-}{\lim}_{r\in R}\big{(}\textrm{$\mathrm{D}$-}{\limsup}_{t\in R}|\langle x_{r+t},x_t\rangle|\big{)}=0
\end{gather*}
over a F{\o}lner sequence $\{F_n\}_1^\infty$ in $(R,+)$. Then
\begin{gather*}
\textrm{$\mathrm{D}$-}{\lim}_{t\in R}x_t=0
\end{gather*}
in $\bH$ under the weak topology (namely, $\mathrm{D}$-$\lim_{t\in R}\langle x_t,x\rangle=0\; \forall x\in\bH$) over $\{F_n\}_1^\infty$.
\end{Lem}

\begin{proof}
Without loss of generality, let $\bH$ is a real Hilbert space.
In contrast to the statement of Lemma~\ref{lem1.4}, we can find some number $\varepsilon$ with $0<\varepsilon<1$ and some unit vector $v\in\bH$ such that
\begin{gather*}
S:=\{t\in R\colon \langle x_t,v\rangle\ge\varepsilon\}\in\mathcal{F}_+\quad \textrm{over }\{F_n\}_1^\infty.
\end{gather*}
We then for $0<\delta<\varepsilon^2$ let
$$
Q=\left\{q\in R\colon\textrm{D-}\limsup_{t\in R}|\langle x_{q+t},x_t\rangle|<\frac{\delta}{2}\right\}.
$$
So $Q$ has density $1$ over $\{F_n\}_1^\infty$ by Def.~\ref{def1.1} and for each $q\in Q$, the Borel set
$$
R_q=\left\{r\in R\colon|\langle x_{q+r},x_r\rangle|<\delta\right\}
$$
has density $1$ over $\{F_n\}_1^\infty$. Applying Lemma~\ref{B} to these sets with $k$ to be specified later. If $r_1,\dotsc,r_k$ satisfy the conclusion of Lemma~\ref{B}, then we shall have that
$$
\langle x_{r_i},v\rangle>\varepsilon,\ 1\le i\le k,\quad \textrm{and}\quad |\langle x_{r_i}, x_{r_j}\rangle|<\delta,\ 1\le i<j\le k.
$$
Set $y_i=x_{r_i}-\varepsilon v$ for $1\le i\le k$. Then
$$
\langle y_i,y_j\rangle<\delta-2\varepsilon+\varepsilon^2<\delta-\varepsilon^2<0,\quad 1\le i<j\le k.
$$
But since the $y_i$ are bounded independently of $k$, and
\begin{align*}
0&\le\|y_1+\dotsm+y_k\|^2=\sum_{i=1}^k\|y_i\|^2+2\sum_{i<j}\langle y_i,y_j\rangle\\
&\le k\max_{1\le i\le k}\|y_i\|^2-k(k-1)\left(\varepsilon^2-\delta\right)
\end{align*}
which tends to $-\infty$ as $k\to\infty$, we thus arrive at a contradiction. 

Therefore the set $S$ in question must have density $0$ over $\{F_n\}_1^\infty$ for any $x$ and any $\varepsilon>0$.
This proves Lemma~\ref{lem1.4}.
\end{proof}

The above proof is different with other versions \cite[p.~445]{BL} and \cite[Theorem~7.11]{EW}. In addition, let $\{\xi_n; n=1,2,\dotsc\}$ be a sequence of independent random variable on a probability space $(\Omega,\mathscr{F},\mathrm{P})$ with $\xi_n\sim N(0,1)$ for all $n\ge1$. Then $\langle\xi_{n+m},\xi_n\rangle=E_\mathrm{P}(\xi_{n+m}\xi_n)=0$ for $m\not=0$ and so this shows that in Lemma~\ref{lem1.4}, we cannot expect that $\textrm{$\mathrm{D}$-}{\lim}_{t\in R}\|x_t\|=0$.
\section{Inverse limits of \textit{Kh}-systems}\label{sec2}
In this section, let $G$ be lcsc $R$-module. Let $(X,\mathscr{X},\mu)$ be any Borel $G$-space and we write $\X=(X,\mathscr{X},\mu,G)$. For any $T\in G$, we write $tT=T^t$ for any $t\in R$. We will consider a kind of multiple dynamics---\textit{Kh}-systems.

Motivated by Khintchine's recurrence theorem and \cite{BHK}, we now introduce the following notation:

\begin{defn}\label{def2.1}
For $\X=(X,\mathscr{X},\mu,G)$, we shall say that $\X$ is a \textit{Kh-system} (\textit{Kh} is for Khintchine), provided that for any integer $l\ge2$,
\begin{itemize}
\item if $f\in \mathfrak{L}^\infty(X,\mathscr{X},\mu)$ with $f\ge0$ \textit{a.e.} and $\int_Xfd\mu>0$, then for any $T_1,\dotsc,T_l\in G$ and any $\varepsilon>0$, the set
\begin{gather*}
\left\{t\in R\,\big{|}\,\int_XT_1^tf\dotsm T_l^tfd\mu>\left(\int_Xfd\mu\right)^l-\varepsilon\right\}
\end{gather*}
is of positive lower Banach density; that is, it is of positive lower density over any weak F{\o}lner sequence $\{F_n\}_1^\infty$ in $(R,+)$.
\end{itemize}
\end{defn}

We can then obtain the following result.

\begin{prop}\label{prop2.2}
Let $\{\mathscr{X}_\theta;\theta\in\Theta\}$ be a totally ordered family of $\sigma$-subalgebras of $\mathscr{X}$ and set $\mathscr{Y}=\sigma\big{(}\bigcup_\theta\mathscr{X}_\theta\big{)}$. If each $(X,\mathscr{X}_\theta,\mu,G)$ is a \textit{Kh}-system, then
$(X,\mathscr{Y},\mu,G)$ is also a \textit{Kh}-system.
\end{prop}

\begin{proof}
Let $f\in \mathfrak{L}^\infty(X,\mathscr{Y},\mu)$ with $f\ge0$ \textit{a.e.} and $\int_Xfd\mu>0$, and let $T_1,\dotsc,T_l\in G$ and $\varepsilon>0$. We may assume $0\le f(x)\le1/2$ \textit{a.e.} without loss of generality.
Let $\epsilon>0$ be any given with $\epsilon<\varepsilon/2$. Since $\mathscr{Y}=\sigma\big{(}\bigcup_\theta\mathscr{X}_\theta\big{)}$ and $\{\mathscr{X}_\theta\}$ is totally ordered, then it follows that one can find some $f^\prime\in\mathfrak{L}^\infty(X,\mathscr{X}_\theta,\mu)$ with $\|f-f^\prime\|_2\ll\epsilon$ for some $\theta\in\Theta$. Simply write $\tilde{f}=E_\mu(f|\mathscr{X}_\theta)$. Then $0\le\tilde{f}\le1/2$ \textit{a.e.} and
$\|f-\tilde{f}\|_2\ll\epsilon$.
Hence we can set $f=\tilde{f}+\phi_{\theta}$ where $|\phi_{\theta}|\le 1$ \textit{a.e.} and $\|\phi_{\theta}\|_2\ll\epsilon$, and further for any $t\in R$
\begin{gather*}
T_1^tf\dotsm T_l^tf=T_1^t\tilde{f}\dotsm T_l^t\tilde{f}+\ldots,\quad \|\ldots\|_2<\epsilon\intertext{and further for any $t\in R$}
\int_XT_1^tf\dotsm T_l^tfd\mu\ge\int_XT_1^t\tilde{f}\dotsm T_l^t\tilde{f}d\mu-\epsilon.
\end{gather*}
Since $(X,\mathscr{X}_\theta,\mu,G)$ is a \textit{Kh}-system and $\int_Xfd\mu=\int_X\tilde{f}d\mu>0$, then the set
\begin{gather*}
\left\{t\in R\,\big{|}\,\int_XT_1^t\tilde{f}\dotsm T_l^t\tilde{f}d\mu>\left(\int_X\tilde{f}d\mu\right)^l-\epsilon\right\}
\end{gather*}
is of positive lower density over any weak F{\o}lner sequence $\{F_n\}_1^\infty$ in $(R,+)$. This implies the statement by Def.~\ref{def2.1} for $\epsilon<\varepsilon/2$.

The proof of Proposition~\ref{prop2.2} is therefore completed.
\end{proof}

This result immediately leads to the following important fact.

\begin{cor}\label{cor2.3}
Let $\{\X_\xi\}$ be a Furstenberg factors chain of $\X=(X,\mathscr{X},\mu,G)$ in accordance with Theorem~\ref{thm0.2}. Assume that the ordinal $\xi$ is a limit ordinal and that $\X_\theta$ is a \textit{Kh}-system for each $\theta<\xi$. Then
$\X_\xi$ is a \textit{Kh}-system.
\end{cor}

\begin{proof}
According to Theorem~\ref{thm0.2}, $\{\mathscr{X}_\theta; \theta<\xi\}$ is totally ordered. Then the statement follows from the foregoing Proposition~\ref{prop2.2}.
\end{proof}

Proposition \ref{prop2.2} shows that the \textit{Kh}-property survives taking inverse limits.

\section{\textit{Kh}-property of totally relatively weak-mixing extensions}\label{sec3}
Let $G$ be an lcsc $R$-module with the identity $e$ and with the continuous scalar multiplication $(t,g)\mapsto g^t$ of $R\times G$ to $G$.
Let there be any given a short factors series:
\begin{gather*}
\X=(X,\mathscr{X},\mu,G)\xrightarrow[]{\textit{Id}_X}\X^\prime=(X,\mathscr{X}^\prime,\mu,G)\xrightarrow[]{\pi}\Y=(Y,\mathscr{Y},\nu,G),
\end{gather*}
where $(X,\mathscr{X},\mu)$ is a standard Borel $G$-space so we can decompose $\mu=\int_Y\mu_yd\nu(y)$ and where $\Y=(Y,\mathscr{Y},\nu,G)$.
Based on the measure-preserving system $\X=(X,\mathscr{X},\mu,G)$, we regard every element $g$ of $G$ as a $\mu$-preserving transformation of $X$ onto itself and we also often identify $g$ with the unitary operator $U_g\colon\phi\mapsto\phi\circ g$ of the $\mathfrak{L}^p$-function spaces if no confusion, for $p\ge1$.

\subsection{Relatively weak-mixing extensions}\label{sec3.1}
Recall from the viewpoint of group as in \cite{FK, Fur} that an extension $\pi\colon\X\rightarrow\Y$ is referred to as \textit{relatively weak-mixing for an element $g$} in $G$ if every $g$-invariant (or equivalently $\{g^n\,|\,n\in\mathbb{Z}\}$-invariant) function $H$ in $\mathfrak{L}^2(X,\mathscr{X},\mu)\otimes_{\Y}\mathfrak{L}^2(X,\mathscr{X},\mu)$ is a function on $(Y,\mathscr{Y},\nu)$ via the relative-product factor $G$-map
\begin{gather*}
{\pi\times_{\Y}\pi}\colon X\times X\rightarrow Y;\quad (x_1,x_2)\mapsto x_1.
\end{gather*}
Then we say $\pi\colon\X\rightarrow\Y$ is \textit{relatively weak-mixing for $G$} if $\pi\colon\X\rightarrow\Y$ is relatively weak-mixing for every $g, g\not=e$, in $G$ (cf.~\cite[Def.~6.3]{Fur}).

Now for any subset $S$ of the $R$-module $G$, by $\langle S\rangle_R$ we denote the submodule generated by the elements $g\in S$ over the ring $(R,+,\cdot)$. We note that for any $H\in\mathfrak{L}^2(X,\mathscr{X},\mu)\otimes_{\Y}\mathfrak{L}^2(X,\mathscr{X},\mu)$, it is $g$-invariant if and only if it is $\langle g\rangle_\mathbb{Z}$-invariant. However, this is not the case about $\langle g\rangle_R$-invariance when $g\not=e$ in general.
More generally than the above relatively weak-mixing in the literature, we now will introduce other relatively weak-mixing condition for $R$-modules.

\begin{defn}[\cite{Dai-pre}]\label{def3.1}
Let $\varGamma$ be a subset of the \textit{lcsc} $R$-module $G$ with $\varGamma\not=\{e\}$ and $g\in G, g\not=e$. Then $\pi\colon\X^\prime\rightarrow\Y$ is said to be
\begin{enumerate}
\item[$(a)$] \textit{relatively weak-mixing for $g$} if every $\langle g\rangle_R$-invariant $H\in\mathfrak{L}^2(X,\mathscr{X}^\prime,\mu)\otimes_{\Y}\mathfrak{L}^2(X,\mathscr{X}^\prime,\mu)$ is a function on $(Y,\mathscr{Y},\nu)$ via ${\pi\times_{\Y}\pi}\colon X\times X\rightarrow Y$.
\item[$(b)$] \textit{totally relatively weak-mixing for $\varGamma$} if ${\pi}\colon\X^\prime\rightarrow\Y$ is relatively weak-mixing for each $g$ in $\varGamma$ with $g\not=e$ under the sense of $(a)$;

\item[$(c)$] \textit{jointly relatively ergodic for $\varGamma$} if every $\langle\varGamma\rangle_R$-invariant function $\varphi\in\mathfrak{L}^2(X,\mathscr{X}^\prime,\mu)$ is a function on $(Y,\mathscr{Y},\nu)$ via ${\pi}\colon\X^\prime\rightarrow\Y$.

\item[$(d)$] \textit{jointly relatively weak-mixing for $\varGamma$} if ${\pi\times_{\Y}\pi}\colon\X^\prime\times_{\Y}\X^\prime\rightarrow\Y$ is jointly relatively ergodic for $\varGamma$ in the sense of $(c)$.
\end{enumerate}
\end{defn}

Since every $\langle\varGamma\rangle_R$-invariant function is $\langle g\rangle_R$-invariant for each $g\in\varGamma$, a totally relatively weak-mixing extension must be a jointly relatively weak-mixing extension for $\varGamma$; but the converse does not need to be true by considering a one-point factor $G$-system $\Y$.

In addition, if $\pi\colon\X^\prime\rightarrow\Y$ is relatively weak-mixing for $g$, then it is not necessarily totally relatively weak-mixing for $\langle g\rangle_R$ unless $rR=R$ for each $r\not=0$ like $R$ to be a field.

When $(R,+,\cdot)=(\mathbb{Z},+,\cdot)$, Def.~\ref{def3.1}-$(a)$ and $(b)$ coincide exactly with the classical case in the literature~\cite{FK, Fur}. However, if $(R,+,\cdot)\not=(\mathbb{Z},+,\cdot)$, then our relatively weak-mixing is weaker than that of \cite{FK, Fur}.

\begin{Lem}\label{lem3.2}
If $\pi\colon\X^\prime=(X,\mathscr{X}^\prime,\mu,G)\rightarrow\Y=(Y,\mathscr{Y},\nu,G)$ is a relatively ergodic extension for an element $T$ in $G$, then for any $\varphi\in \mathfrak{L}^2(X,\mathscr{X}^\prime,\mu)$ with $E_\mu(\varphi|\pi^{-1}[\mathscr{Y}])=0$ (or equivalently $E_\mu(\varphi|\Y)=0$), we have
\begin{equation*}
\lim_{n\to\infty}\frac{1}{|F_n|}\int_{F_n}T^t\varphi\,dt=0
\end{equation*}
in $\mathfrak{L}^2$-norm, over any weak F{\o}lner sequence $\{F_n\}_1^\infty$ in $(R,+)$.
\end{Lem}

\begin{proof}
Let $\mathscr{X}_{\mu,\langle T\rangle_R}^\prime=\{A\in\mathscr{X}^\prime\,|\,\mu(T^{-t}(A)\vartriangle A)=0\ \forall t\in R\}$. By the $\mathfrak{L}^2$-mean ergodic theorem (cf.~\cite{EW} or more precisely \cite{Dai-16}), it follows that
\begin{equation*}
\mathfrak{L}^2(\mu)\textit{-}\lim_{n\to\infty}\frac{1}{|F_n|}\int_{F_n}{T^t}\varphi\,dt=E_\mu\left(\varphi|\mathscr{X}_{\mu,\langle T\rangle_R}^\prime\right).
\end{equation*}
By hypothesis $\pi\colon\X^\prime\rightarrow\Y$ is relatively ergodic for the element $T$ so that $\mathscr{X}_{\mu,\langle T\rangle_R}^\prime\subseteq\pi^{-1}[\mathscr{Y}]$ and then
$E_\mu\left(\varphi|\mathscr{X}_{\mu,\langle T\rangle_R}^\prime\right)=E_\mu\big{(}E_\mu(\varphi|\pi^{-1}[\mathscr{Y}])|\mathscr{X}_{\mu,\langle T\rangle_R}^\prime\big{)}\equiv0$ \textit{a.e.} which yields the desired result.
\end{proof}

This lemma is a generalization of \cite[Lemma~6.1]{Fur}. The following consequence of it is a generalization of \cite[Lemma~1.3]{FK} or \cite[Proposition~6.2]{Fur} (also see \cite[Proposition~7.30]{EW}).

\begin{prop}\label{prop3.3}
If $\pi\colon\X^\prime\rightarrow\Y$ is relatively weak-mixing for some $T$, where $T\in G$ with $T\not=e$, and $\phi,\psi\in \mathfrak{L}^\infty(X,\mathscr{X}^\prime,\mu)$, then
$$
\lim_{n\to\infty}\frac{1}{|F_n|}\int_{F_n}\int_Y\left|E_\mu(\psi T^t\phi|\Y)(y)-E_\mu(\psi|\Y) T^tE_\mu(\phi|\Y)(y)\right|^2\nu(dy)dt=0,
$$
over any weak F{\o}lner sequence $\{F_n\}_1^\infty$ in $(R,+)$.
\end{prop}

\begin{proof}
Take $\phi,\psi\in \mathfrak{L}^\infty(X,\mathscr{X}^\prime,\mu)$ so that $\phi\otimes\bar{\phi},\psi\otimes\bar{\psi}\in \mathfrak{L}^2(X\times X,\mathscr{X}^\prime\otimes_{\Y}\mathscr{X}^\prime,\mu\otimes_{\Y}\mu)$. By \cite[Proposition~5.12]{Fur},
\begin{align*}
E_{\mu\otimes_{\Y}\mu}\big{(}(\psi\otimes\bar{\psi})T^t(\phi\otimes\bar{\phi})|\Y\big{)}&=E_{\mu\otimes_{\Y}\mu}\big{(}(\psi T^t\phi)\otimes(\overline{\psi T^t\phi})|\Y\big{)}\\
&=\left|E_\mu(\psi T^t\phi|\Y)\right|^2.
\end{align*}
First assume that $E_\mu(\phi|\Y)=0$, then $E_{\mu\otimes_{\Y}\mu}(\phi\otimes\bar{\phi}|\Y)=\left|E_\mu(\phi|\Y)\right|^2=0$. Then by Lemma~\ref{lem3.2} and the above equality, we can obtain
$$
\lim_{n\to\infty}\frac{1}{|F_n|}\int_{F_n}\int_Y\left|E_\mu(\psi T^t\phi|\Y)(y)\right|^2\nu(dy)dt=0.
$$
Now let $\phi$ be an arbitrary bounded function. Then the function $\hat{\phi}=\phi-E_\mu(\phi|\pi^{-1}[\mathscr{Y}])$ satisfies $E_\mu(\hat{\phi}|\Y)=0$. Replacing $\phi$ by $\hat{\phi}$ in the above equality leads to what we need from that
$$
E_\mu\big{(}\psi T^tE_\mu(\phi|\pi^{-1}[\mathscr{Y}])|\Y\big{)}=E_\mu(\psi|\Y)T^tE_\mu(\phi|\Y).
$$
This proves Proposition~\ref{prop3.3}.
\end{proof}

The following is a generalization of \cite[Proposition~1.2]{FK} or \cite[Proposition~6.3]{Fur}.

\begin{Lem}\label{lem3.4}
If $\pi\colon\X^\prime\rightarrow\Y$ is relatively weak-mixing for some $T$, where $T\in G, T\not=e$, and $\pi_2\colon\X_2^\prime\rightarrow\Y$ is relatively ergodic for $T$, then $\pi\times_{\Y}\pi_2\colon\X^\prime\times_{\Y}\X_2^\prime\rightarrow\Y$ is relatively ergodic for $T$. Particularly, $\pi\times_{\Y}\pi\colon\X^\prime\times_{\Y}\X^\prime\rightarrow\Y$ is relatively weak-mixing for $T$.
\end{Lem}

\begin{proof}
The statement follows from a slight modification of the proof of \cite[Proposition~1.2]{FK}. So we omit the details here.
\end{proof}

The above arguments are still valid for the jointly relatively weak-mixing case. Specially we have the following

\begin{prop}\label{prop3.5}
If $\pi\colon\X^\prime\rightarrow\Y$ is jointly relatively weak-mixing for $G$ itself, and $\phi,\psi\in \mathfrak{L}^\infty(X,\mathscr{X}^\prime,\mu)$, then
$$
\lim_{n\to\infty}\frac{1}{m_G(F_n)}\int_{F_n}\int_Y\left|E_\mu(\psi g\phi|\Y)-E_\mu(\psi|\Y)gE_\mu(\phi|\Y)\right|^2d\nu dg=0,
$$
over any weak F{\o}lner sequence $\{F_n\}_1^\infty$ in $G$, where $m_G$ and $dg$ are the same left Haar measure of the lcsc abelian group $G$.
\end{prop}

\subsection{\textit{Kh}-property of totally relatively weak-mixing extensions}\label{sec3.2}
We will now consider the lifting of \textit{Kh}-property of totally relatively weak-mixing extensions. Following Def.~\ref{def3.1} and Def.~\ref{def1.1} we begin with a lemma, which generalizes \cite[Lemmas~7.3 and 7.6 for $R=\mathbb{Z}$]{Fur}. Also see \cite[Proposition~7.30]{EW} for $G=R=\mathbb{Z}$.

\begin{Lem}\label{lem3.6}
If $\pi\colon\X^\prime=(X,\mathscr{X}^\prime,\mu,G)\rightarrow\Y=(Y,\mathscr{Y},\nu,G)$
is relatively weak-mixing for some $T$ where $T\in G$ with $T\not=e$, and if $\psi\in\mathfrak{L}^2(X,\mathscr{X}^\prime,\mu)$ and $\phi$ belongs to $\mathfrak{L}^\infty(X,\mathscr{X}^\prime,\mu)$ with $E_\mu(\phi|\pi^{-1}[\mathscr{Y}])=0$ or equivalently $E_\mu(\phi|\Y)=0$, then
\begin{equation*}
\textrm{$\mathrm{D}$-}{\lim}_{t\in R}\int_X\psi T^t\phi d\mu=0\quad \textit{and}\quad \textrm{$\mathrm{D}$-}{\lim}_{t\in R}\|E_\mu(\psi T^t\phi|\Y)\|_{2,\nu}=0
\end{equation*}
over any weak F{\o}lner sequence $\{F_n\}_1^\infty$ in $(R,+)$.
\end{Lem}

\begin{proof}
Over any weak F{\o}lner sequence $\{F_n\}_1^\infty$ in $(R,+)$, according to Proposition~\ref{prop3.3} we have, for any $\psi\in\mathfrak{L}^\infty(X,\mathscr{X}^\prime,\mu)$,
$$
\lim_{n\to\infty}\frac{1}{|F_n|}\int_{F_n}\int_Y\left|E_\mu(\psi T^t\phi|\Y)(y)\right|^2\,d\nu(y)dt=0
$$
so that
\begin{align*}
0=\lim_{n\to\infty}\frac{1}{|F_n|}\int_{F_n}\left|\int_YE_\mu(\psi T^t\phi|\Y)d\nu\right|^2dt=\lim_{n\to\infty}\frac{1}{|F_n|}\int_{F_n}\left|\int_X\psi T^t\phi d\mu\right|^2dt.
\end{align*}
This implies that $\textrm{$\mathrm{D}$-}{\lim}_{t\in R}\int_X\psi T^t\phi d\mu=0$ and $\textrm{$\mathrm{D}$-}{\lim}_{t\in R}\|E_\mu(\psi T^t\phi|\Y)\|_{2,\nu}=0$ over $\{F_n\}_1^\infty$ for a dense set of $\psi$.

Thus the proof of Lemma~\ref{lem3.6} is completed.
\end{proof}

The following result is important for us to prove that a totally relatively weak-mixing extension of a \textit{Kh}-system is also a \textit{Kh}-system, which is a generalization of \cite[Proposition~7.4]{Fur} (also see \cite[Theorem~7.27 and Proposition~7.31]{EW} for the special case that $G=R=\mathbb{Z}$).

\begin{prop}\label{prop3.7}
Let $\pi\colon\X^\prime\rightarrow\Y$ be totally relatively weak-mixing for $G$ and let $T_1,\dotsc,T_l\in G$ be distinct with $T_i\not=e$ for $1\le i\le l$. If
$f_1,\dotsc,f_l\in\mathfrak{L}^\infty(X,\mathscr{X}^\prime,\mu)$, then in $\mathfrak{L}^2(X,\mathscr{X}^\prime,\mu)$ with the \textbf{\textit{weak}} topology, it holds that
\begin{gather*}
\mathrm{D}\textit{-}\lim_{t\in R} \left\{\prod_{i=1}^lT_i^tf_i-\prod_{i=1}^lT_i^tE_\mu(f_i|\pi^{-1}[\mathscr{Y}])\right\}=0
\end{gather*}
over any weak F{\o}lner sequence in $(R,+)$.
\end{prop}

\begin{proof}
Given any distinct elements $T_1,\dotsc,T_l\in G$ with $T_i\not=e$ for $1\le i\le l$ and let $\{F_n\}_1^\infty$ be an arbitrary weak F{\o}lner sequence in $(R,+)$. We will prove by induction on $l$ that over this weak F{\o}lner sequence,
\begin{equation}\label{eq3.1}
 \prod_{i=1}^lT_i^tf_i-\prod_{i=1}^lT_i^tE_\mu(f_i|\pi^{-1}[\mathscr{Y}])\xrightarrow[]{\textit{weakly, $\mathrm{D}$-}}0
\end{equation}
for any $f_1,\dotsc,f_l\in\mathfrak{L}^\infty(X,\mathscr{X}^\prime,\mu)$.

We first consider the case of $l=1$. Let $\phi=f_1-E_\mu(f_1|\pi^{-1}[\mathscr{Y}])$ and set
$\phi_t=T_1^t\phi$. Then by Lemma~\ref{lem3.6}, it follows that
\begin{gather*}
\mathrm{D}\textrm{-}\lim_{r\in R}\left(\mathrm{D}\textrm{-}\lim_{t\in R}|\langle\phi_{r+t},\phi_t\rangle|\right)=\mathrm{D}\textrm{-}\lim_{r\in R}\left(\mathrm{D}\textrm{-}\lim_{t\in R}\left|\int_X\phi T_1^r\phi d\mu\right|\right)=0.
\end{gather*}
Thus $\mathrm{D}\textrm{-}\lim_{t\in R}\{T_1^tf_1-T_1^tE_\mu(f_1|\pi^{-1}[\mathscr{Y}])\}=\mathrm{D}\textrm{-}\lim_{t\in R}\phi_t=0$ weakly
by Lemma~\ref{lem1.4}.

Next, we assume $T_1,\dotsc,T_l$ are distinct elements in $G$ with $T_i\not=e$ for $1\le i\le l$ and $l\ge2$.
By applying to the foregoing expression (\ref{eq3.1}) the telescoping sum identity
\begin{equation}\label{eq3.2}
\prod_{i=1}^la_i-\prod_{i=1}^lb_i=\sum_{j=1}^l\left(\prod_{i=1}^{j-1}a_i\right)(a_j-b_j)\left(\prod_{i=j+1}^{l}b_i\right)
\end{equation}
where and in the sequel an empty product is always interpreted as $1$, we reduce our proof of the above (\ref{eq3.1}) to the proof of
\begin{equation*}
\textit{weakly }\textrm{D-}\lim_{t\in R}{\prod}_{i=1}^lT_i^tf_i=0,\quad \forall f_1,\dotsc,f_l\in\mathfrak{L}^\infty(X,\mathscr{X}^\prime,\mu)
\end{equation*}
under the additional hypothesis that for some $1\le j\le l$, $E_\mu(f_j|\pi^{-1}[\mathscr{Y}])=0$ since we have
$$E_\mu\left(f_j-E_\mu(f_j|\pi^{-1}[\mathscr{Y}])\right)=0.$$
Without loss of generality assume $E_\mu(f_l|\pi^{-1}[\mathscr{Y}])=0$, and the van der Corput-type lemma (Lemma~\ref{lem1.4}) may be employed as follows. Set
\begin{equation*}
\psi_t=\prod_{i=1}^lT_i^tf_i\in\mathfrak{L}^2(X,\mathscr{X}^\prime,\mu)\quad \forall t\in R.
\end{equation*}
Then for any $r,t\in R$, we can get that
\begin{align*}
\langle \psi_{r+t},\psi_t\rangle&=\int_X{\prod}_{i=1}^lT_i^{r+t}f_i{\prod}_{i=1}^lT_i^tf_id\mu=\int_X(f_lT_l^rf_l){\prod}_{i=1}^{l-1}(T_iT_l^{-1})^t(f_iT_i^rf_i)d\mu.
\end{align*}
Since $T_1,T_2,\dotsc,T_l$ are distinct, we can apply the induction hypothesis to obtain
$$
\mathrm{D}\textrm{-}\lim_{t\in R}\left\{\langle \psi_{r+t},\psi_t\rangle-\int_X(f_lT_l^rf_l)\prod_{i=1}^{l-1}(T_iT_l^{-1})^tE_\mu(f_iT_i^rf_i|\pi^{-1}[\mathscr{Y}])d\mu\right\}=0
$$
so that
$$
\mathrm{D}\textrm{-}\limsup_{t\in R}|\langle \psi_{r+t},\psi_t\rangle|=\mathrm{D}\textrm{-}\limsup_{t\in R}\left|\int_X(f_lT_l^rf_l)\prod_{i=1}^{l-1}(T_iT_l^{-1})^tE_\mu(f_iT_i^rf_i|\pi^{-1}[\mathscr{Y}])d\mu\right|.
$$
Then by the self-adjointness of conditional expectation (cf.~\cite[Theorem~6.1(vi)]{Kal}), we have
\begin{align*}
&\int_X(f_lT_l^rf_l)\prod_{i=1}^{l-1}(T_iT_l^{-1})^tE_\mu(f_iT_i^rf_i|\pi^{-1}[\mathscr{Y}])d\mu\\
&=\int_XE_\mu(f_lT_l^rf_l|\pi^{-1}[\mathscr{Y}])\prod_{i=1}^{l-1}(T_iT_l^{-1})^tE_\mu(f_iT_i^rf_i|\pi^{-1}[\mathscr{Y}])d\mu.
\end{align*}
Since $\big{\|}{\prod}_{i=1}^{l-1}T_i^tT_l^{-t}(f_iT_i^rf_i)\big{\|}_\infty\le C$ for some constant $C>0$, then by the H\"{o}lder inequality it follows that
$$
\left|\int_X(f_lT_l^rf_l)\prod_{i=1}^{l-1}(T_iT_l^{-1})^tE_\mu(f_iT_i^rf_i|\pi^{-1}[\mathscr{Y}])d\mu\right|\le C\|E_\mu(f_lT_l^rf_l|\Y)\|_{2,\nu}
$$
so that
\begin{gather*}
\mathrm{D}\textrm{-}\limsup_{t\in R}|\langle \psi_{r+t},\psi_t\rangle|\le C\|E_\mu(f_lT_l^rf_l|\Y)\|_{2,\nu}\quad \forall r\in R\intertext{and further by Lemma~\ref{lem3.6}}
\mathrm{D}\textit{-}\lim_{r\in R}\left(\mathrm{D}\textit{-}\limsup_{t\in R}|\langle\psi_{r+t},\psi_t\rangle|\right)\le \mathrm{D}\textit{-}\lim_{r\in R}C\|E_\mu(f_lT_l^rf_l|\Y)\|_{2,\nu}=0.
\end{gather*}
Hence Lemma~\ref{lem1.4} follows that $\textrm{D-}\lim_{t\in R}\psi_t=0$ over~$\{F_n\}_1^\infty$ in the sense of the weak topology of $\mathfrak{L}^2(X,\mathscr{X}^\prime,\mu)$.

The proof of Proposition~\ref{prop3.7} is thus completed.
\end{proof}

The following result is a generalization of the classical $\mathbb{Z}$-module case (cf.~\cite[Theorem~2.3]{F77}, \cite[Theorem~4.10]{Fur} and \cite[Theorem~4.5]{B} by using different induction approaches).

\begin{cor}\label{cor3.8}
Let $\X^\prime=(X,\mathscr{X}^\prime,\mu,G)$ be totally weak-mixing. Then for any distinct elements $T_1,\dotsc,T_l$ in $G$ with $T_i\not=e$ for $1\le i\le l$ and any $f_1,\dotsc,f_l$ in $\mathfrak{L}^\infty(X,\mathscr{X}^\prime,\mu)$ we have
$$
\textrm{$\mathrm{D}$-}\lim_{t\in R}\prod_{i=1}^lT_i^tf_i=\left(\int_Xf_1d\mu\right)\dotsm\left(\int_Xf_ld\mu\right)\quad \textrm{weakly in } \mathfrak{L}^2(X,\mathscr{X},\mu)
$$
over any weak F{\o}lner sequence in $(R,+)$.
\end{cor}

\begin{proof}
Since $\X$ is a totally relatively weak-mixing extension of the one-point system for $G$, hence the statement follows from Proposition~\ref{prop3.7}.
\end{proof}

As a consequence of Proposition~\ref{prop3.7} we have the following, which is a generalization of \cite[Proposition~7.7]{Fur}.

\begin{Lem}\label{lem3.9}
If $\pi\colon\X^\prime\rightarrow\Y$ is totally relatively weak-mixing for $G$ and $T_1,\dotsc,T_l\in G$ are distinct with $T_i\not=e$ for $1\le i\le l$, then
for any $f_0\in\mathfrak{L}^2(X,\mathscr{X}^\prime,\mu)$ and $f_1,\dotsc,f_l\in\mathfrak{L}^\infty(X,\mathscr{X}^\prime,\mu)$,
\begin{equation*}
\textrm{$\mathrm{D}$-}\lim_{t\in R}\left\{\int_Xf_0{\prod}_{i=1}^lT_i^tf_i\,d\mu-\int_YE_\mu(f_0|\Y){\prod}_{i=1}^lT_i^tE_\mu(f_i|\Y)\,d\nu\right\}=0,
\end{equation*}
over any weak F{\o}lner sequence in $(R,+)$.
\end{Lem}

\begin{proof}
Over any weak F{\o}lner sequence in $(R,+)$, we rewrite this as
\begin{equation*}
\textrm{D-}\lim_{t\in R}\left\{\int_Xf_0{\prod}_{i=1}^lT_i^tf_i\, d\mu-\int_XE_\mu(f_0|\pi^{-1}[\mathscr{Y}]){\prod}_{i=1}^lT_i^tE_\mu(f_i|\pi^{-1}[\mathscr{Y}])\,d\mu\right\}=0
\end{equation*}
and noting that by the self-adjointness of the conditional expectation operator $E_\mu(\centerdot|\pi^{-1}[\mathscr{Y}])$
$$
\int_XE_\mu(f_0|\pi^{-1}[\mathscr{Y}]){\prod}_{i=1}^lT_i^tE_\mu(f_i|\pi^{-1}[\mathscr{Y}])\,d\mu=\int_Xf_0{\prod}_{i=1}^lT_i^tE_\mu(f_i|\pi^{-1}[\mathscr{Y}])\,d\mu,
$$
whence we only need to prove that
$$
\textrm{D-}\lim_{t\in R}\left\langle f_0,\prod_{i=1}^lT_i^tf_i-\prod_{i=1}^lT_i^tE_\mu(f_i|\pi^{-1}[\mathscr{Y}])\right\rangle=0.
$$
However, this follows at once from Proposition~\ref{prop3.7}.
\end{proof}

The following is a simple consequence of Lemma~\ref{lem3.9}.

\begin{prop}\label{prop3.10}
Let $\pi\colon\X^\prime\rightarrow\Y$ be totally relatively weak-mixing for $G$ and let $\Y$ be a \textit{Kh}-system; then $\X^\prime$ is also a \textit{Kh}-system.
\end{prop}

Finally we remark that for a jointly relatively weak-mixing extension $\pi\colon\X^\prime\rightarrow\Y$ (cf.~Def.~\ref{def3.1}-$(d)$) there exists no such an useful theory in the literature.
\section{\textit{Kh}-property of chaotic $G$-systems}\label{sec4}
In this section, let $G$ be an lcsc $R$-module and write the continuous scalar multiplication as
$(t,T)\mapsto T^t$
from $R\times G$ to $G$. Let $\X^\prime=(X,\mathscr{X}^\prime,\mu,G)$ be a (not necessarily standard) Borel $G$-system.
\subsection{Compact $G$-system}\label{sec4.1}
Different with to say that $X$ is a compact space or $\X^\prime$ is relatively compact, by a \textit{compact $G$-space} $(X,\mathscr{X}^\prime,\mu)$ or a \textit{compact $G$-system} $\X^\prime=(X,\mathscr{X}^\prime,\mu, G)$, we mean that for each $f$ in $\mathfrak{L}^2(X,\mathscr{X}^\prime,\mu)$, the $G$-orbit $G[f]=\{Tf\,|\,T\in G\}$ is precompact/totally-bounded in the Hilbert space $\mathfrak{L}^2(X,\mathscr{X}^\prime,\mu)$.
\subsection{Chaotic $G$-system}\label{sec4.2}
We shall say $\X^\prime=(X,\mathscr{X}^\prime,\mu,G)$ is \textbf{\textit{chaotic}} if there exists a direct product $G=G_c\times G_{w}$ of two nontrivial $R$-submodules of $G$ such that $\X_c^\prime=(X,\mathscr{X}^\prime,\mu,G_c)$ is a compact $G_c$-system and $\X_{w}^\prime=(X,\mathscr{X}^\prime,\mu,G_{w})$ is a totally weak-mixing $G_{w}$-system. This is stronger than a primitive system in the Furstenberg factors chain of $\X$.

The following result will play an important role in proving Theorem~\ref{Kh}.

\begin{prop}\label{prop4.1}
If $\X^\prime=(X,\mathscr{X}^\prime,\mu,G)$ is chaotic, then it is a \textit{Kh}-system such that for any $f$ in $\mathfrak{L}^\infty(X,\mathscr{X}^\prime,\mu)$ with $f\ge0$ \textit{a.e.} and $\int_Xfd\mu>0$ and for any $\varepsilon>0$,
\begin{gather}\label{eq4.1}
\liminf_{n\to+\infty}\frac{1}{|F_n|}\int_{F_n}\left(\int_XT_1^tf\dotsm T_l^tf\,d\mu\right)dt>\left(\int_Xfd\mu\right)^l-\varepsilon
\end{gather}
over any weak F{\o}lner sequence $\{F_n\}_1^\infty$ in $(R,+)$, for any $T_1,\dotsc,T_l\in G$, $l\ge2$.
\end{prop}

\begin{proof}
Let $f\in\mathfrak{L}^\infty(X,\mathscr{X}^\prime,\mu)$ be such that $f\ge0$ \textit{a.e.} and $\int_Xfd\mu>0$. Given any weak F{\o}lner sequence $\{F_n\}_1^\infty$ in $(R,+)$ and any
$T_1,\dotsc,T_l\in G$ where $l\ge2$. Since $G=G_c\times G_{w}$, we may write
\begin{gather*}
T_i=S_iH_i,\quad S_i\in G_{w},\ H_i\in G_c,\quad i=1,\dotsc,l.
\end{gather*}
Moreover, there is no loss of generality in assuming
$$
0\le f(x)\le\frac{1}{2}\quad (\mu\textit{-a.e.}).
$$
As $\X_{w}^\prime=(X,\mathscr{X}^\prime,\mu,G_{w})$ is totally weak-mixing, it follows from Corollary~\ref{cor3.8} and Jensen's inequality that
\begin{gather*}
\mathrm{D}\textit{-}\lim_{t\in R}\int_XS_1^tf\dotsm S_l^tf\,d\mu>\left(\int_Xfd\mu\right)^l-\frac{\varepsilon}{2}.
\end{gather*}
Thus, the set
$$
\mathcal{S}=\left\{t\in R\,\big{|}\int_XS_1^tf\dotsm S_l^tf\,d\mu>\left(\int_Xfd\mu\right)^l-\frac{\varepsilon}{2}\right\}
$$
is of density $1$ over $\{F_n\}_1^\infty$.

On the other hand, since $\X_c^\prime=(X,\mathscr{X}^\prime,\mu,G_c)$ is a compact $G_c$-system, it follows that for arbitrary $\varepsilon>0$,
$$
\mathcal{H}=\left\{t\in R\,\big{|}\,\max_{1\le i\le l}\|H_i^tf-f\|_2<\frac{\varepsilon}{2^{l+1}}\right\}
$$
is of positive lower density over $\{F_n\}_1^\infty$. Indeed, since $f(x)$ is almost periodic for $G_c$, the orbits $G_c[f]$ is precompact in $\mathfrak{L}^2(X,\mathscr{X}^\prime,\mu)$. Hence the partial orbits $\langle H_i\rangle_R[f]$, for $1\le i\le l$, all are also precompact in $\mathfrak{L}^2(X,\mathscr{X}^\prime,\mu)$.
Write
$$W=\stackrel{l\textit{-times}}{\overbrace{\mathfrak{L}^2(X,\mathscr{X}^\prime,\mu)\times\dotsm\times\mathfrak{L}^2(X,\mathscr{X}^\prime,\mu)}}.$$
By considering the diagonal $R$-action
$\varphi\colon R\times W\rightarrow W$
given by
\begin{gather*}
(t,(\phi_1,\dotsc,\phi_l))\mapsto(H_1^t\phi_1,\dotsc,H_l^t\phi_l),
\end{gather*}
and using Tychonoff's Theorem it follows that the $l$-tuple partial orbit of the point $(f,\dotsc,f)$
$$
R_\varphi[(f,\dots,f)]=\left\{(H_1^tf,\dotsc,H_l^tf)\in W\,|\, t\in R\right\}
$$
is precompact in $W$ and thus one can find a finite set of points, say $t_1, \dotsc, t_k$, in $R$ such that
$$
\min_{1\le j\le k}\max_{1\le i\le l}\|H_i^tf-H_i^{t_j}f\|_2<\frac{\varepsilon}{2^{l+1}}\quad \forall t\in R
$$
so that the Borel set
\begin{equation}\label{eq4.2}
\mathcal{H}=\left\{t\in R\,\big{|}\, \max_{1\le i\le l}\|H_i^tf-f\|_2<\frac{\varepsilon}{2^{l+1}}\right\}
\end{equation}
is \textbf{\textit{strongly syndetic}} in $(R,+)$ in the sense that $\mathcal{H}+\{t_1,\dotsc,t_k\}=R$.
This implies that $\mathcal{H}$ has positive lower density over $\{F_n\}_1^\infty$ in $(R,+)$; this is because otherwise, there would be a subsequence of $\{n\}$, say $n_\ell\to\infty$, such that
\begin{align*}
0&=\lim_{\ell\to\infty}\frac{|\mathcal{H}\cap F_{n_\ell}|}{|F_{n_\ell}|}=\lim_{\ell\to\infty}\frac{|\mathcal{H}\cap(F_{n_\ell}-t_j)|}{|F_{n_\ell}|}
=\lim_{\ell\to\infty}\frac{|(\mathcal{H}+t_j)\cap F_{n_\ell}|}{|F_{n_\ell}|}\quad (1\le j\le k)\\
&=\lim_{\ell\to\infty}\frac{|(\mathcal{H}+K)\cap F_{n_\ell}|}{|F_{n_\ell}|}\quad(\textrm{where }K=\{t_1,\dotsc,t_k\})\\
&=1
\end{align*}
which is a contradiction to $\mathcal{H}+\{t_1,\dotsc,t_k\}=R$.

Then for $1\le i\le l$, we can express $H_i^tf$ as
$$
H_i^tf(x)=f(x)+\psi_{i,t}(x)\ \forall t\in\mathcal{H},\quad \textit{where }|\psi_{i,t}(x)|\le 1,\ \|\psi_{i,t}\|_2<\frac{\varepsilon}{2^{l+1}}.
$$
Next we define $P=\mathcal{S}\cap\mathcal{H}$ which is of positive lower density over $\{F_n\}_1^\infty$. Then for any $t\in P$,
\begin{equation*}\begin{split}
\int_XT_1^tf\dotsm T_l^tf\,d\mu&=\int_XS_1^t(H_1^tf)\dotsm S_l^t(H_l^tf)\,d\mu\\
&\ge\int_XS_1^tf\dotsm S_l^tf\,d\mu-\frac{\varepsilon}{2}\\
&>\left(\int_Xfd\mu\right)^l-\varepsilon.
\end{split}\end{equation*}
This implies (\ref{eq4.1}) for $\varepsilon>0$ is arbitrary.

This proves Proposition~\ref{prop4.1}.
\end{proof}

\subsection{Proof of Theorem~\ref{Kh}}
Let $G$ be an lcscN $R$-module over a syndetic ring $(R,+,\cdot)$ and let $(X,\mathscr{X},\mu)$ be a nontrivial standard Borel $G$-space.
Let $\X_1=(X,\mathscr{X}_1,\mu,G)$ be the first intermediate factor in the Furstenberg factors chain of $\X=(X,\mathscr{X},\mu,G)$ by Theorem~\ref{thm0.2}.
Then by Proposition~\ref{prop4.1}, it follows that for any
$A\in\mathscr{X}_1$ with $\mu(A)>0$, any $g_1,\dotsc,g_l\in G$ with any $l\ge2$ and any $\varepsilon>0$, the set
\begin{gather*}
N_{g_1,\dotsc,g_l}(A,\varepsilon)=\left\{t\in R\,\big{|}\,\int_Xg_1^t1_A\dotsm g_l^t1_A\,d\mu>\mu(A)^l-\varepsilon\right\},
\end{gather*}
is of positive lower Banach density in $(R,+)$; i.e.,
\begin{gather}\label{eq4.3}
\liminf_{n\to+\infty}\frac{|F_n\cap N_{g_1,\dotsc,g_l}(A,\varepsilon)|}{|F_n|}>0
\end{gather}
over any weak F{\o}lner sequence $\{F_n\}_{1}^\infty$ in $(R,+)$.

We now only need to show that $N_{g_1,\dotsc,g_l}(A,\varepsilon)$ is syndetic in $(R,+)$; that is, there is a compact subset $K$ of $R$ with $N_{g_1,\dotsc,g_l}(A,\varepsilon)+K=R$. To the contrary, for any weak F{\o}lner sequence $\{K_n\}_{1}^\infty$ in $(R,+)$, there exists a sequence of elements $t_n\in R$ so that $N_{g_1,\dotsc,g_l}(A,\varepsilon)\cap(t_n+K_n)=\varnothing$ for all $n\ge1$. Set $F_n=t_n+K_n$ and it is easy to see that $\{F_n\}_{1}^\infty$ is also a weak F{\o}lner sequence in $(R,+)$. But
\begin{gather*}
\int_Xg_1^t1_A\dotsm g_l^t1_A\,d\mu\le\mu(A)^l-\varepsilon\quad \forall t\in F_n, \ n\ge1.
\end{gather*}
This is a contradiction to (\ref{eq4.3}) and thus the proof of Theorem~\ref{Kh} is completed.

It should be noted that if (\ref{eq4.3}) holds only over the F{\o}lner sequence $\{F_n\}_{1}^\infty$ satisfying the Tempelman or Shulman conditions, then we cannot deduce the syndetic property of $N_{g_1,\dotsc,g_l}(A,\varepsilon)$ in $(R,+)$ since $F_n=t_n+K_n$ in the above is not necessarily a F{\o}lner sequence satisfying these conditions.
\section*{\textbf{Acknowledgments}}%
This work was partly supported by National Natural Science Foundation of China grant $\#$11271183 and PAPD of Jiangsu Higher Education Institutions.



\end{document}